\documentclass[journal,onecolumn]{IEEEtran}

\usepackage{amsmath,amssymb,amsfonts,amsthm,mathtools}
\usepackage{cite}
\usepackage{booktabs}
\usepackage{array}
\usepackage{enumitem}
\usepackage{url}
\usepackage[hidelinks]{hyperref}
\allowdisplaybreaks

\newtheorem{theorem}{Theorem}
\newtheorem{lemma}{Lemma}
\newtheorem{corollary}{Corollary}
\newtheorem{proposition}{Proposition}
\newtheorem{definition}{Definition}

\newcommand{\Z}{\mathbb Z}

\newcommand{\Q}{\mathcal Q}
\newcommand{\Qi}{\mathcal Q^{\rm irr}}
\newcommand{\eps}{\varepsilon}

\begin{document}

\title{A Local Valuation Criterion for Quadratic-Permutation Interleaved Zadoff--Chu Sequences}

\author{Yutong~Zhang,~\IEEEmembership{Member,~IEEE,}
        and~Yaoran~Yang,~\IEEEmembership{Member,~IEEE}%
\thanks{Manuscript received Month Day, 2026; revised Month Day, 2026.}
\thanks{Yutong Zhang and Yaoran Yang are with the School of Mathematics,
Sichuan University, Chengdu 610065, China
(e-mail: yutongzhang@stu.scu.edu.cn; yangyaoran@stu.scu.edu.cn).}
\thanks{Corresponding author: Yutong Zhang.}}

\markboth{IEEE Transactions on Information Theory,~Vol.~XX, No.~XX, Month~2026}%
{Zhang and Yang: A Local Valuation Criterion for QPP-Interleaved Zadoff--Chu Sequences}

\maketitle

\begin{abstract}
Berggren and Popovi\'c introduced quadratic-permutation-polynomial interleaved Zadoff--Chu sequences and, from exhaustive data, conjectured that, after excluding the affine case, all normalized QPP-interleaved Zadoff--Chu sequences are inequivalent to ordinary Zadoff--Chu sequences precisely for prime-power lengths $N=p^n$ with $p>3$ and $n>1$.  We give an exact local arithmetic criterion.  For a normalized QPP $\pi_{a,b}(k)=ak^2+bk\pmod N$, the interleaved sequence is equivalent, under the standard five CAZAC-preserving operations, to a Zadoff--Chu sequence if and only if, for every prime power $p^\alpha\Vert N$, the valuation of $a$ satisfies
\[
\nu_p(a)\ge
\begin{cases}
0, & p=2,\ \alpha=1,\\
\alpha-1, & p=2,\ \alpha\ge2,\\
\alpha-1, & p=3,\\
\alpha, & p>3.
\end{cases}
\]
The proof is based on a third finite-difference invariant of the lifted Zadoff--Chu phase, namely
\[
\Delta^3\bigl((ak^2+bk+\varepsilon_N+2q)(ak^2+bk)\bigr)
      =12a(2ak+3a+b).
\]
As a consequence, the conjectured prime-power boundary is not correct: the exact non-vacuous condition for all nonzero normalized QPPs to be inequivalent to Zadoff--Chu sequences is that $N$ is odd, $9\nmid N$, and $p^2\mid N$ for at least one prime $p\ge5$.  In particular, $N=75=3\cdot5^2$ is the smallest non-prime-power counterexample to the conjectured ``only if'' direction.  A second corollary records the corresponding statement for irreducible QPPs.
\end{abstract}


\section{Introduction}

Let $N\ge2$ and let
\[
        \zeta_N=\exp(-\pi\sqrt{-1}/N),\qquad
        W_N=\zeta_N^2=\exp(-2\pi\sqrt{-1}/N).
\]
For $u\in\Z$ with $(u,N)=1$, $q\in\Z$, and
\[
        \eps_N=N\bmod 2\in\{0,1\},
\]
the length-$N$ Zadoff--Chu phase convention used here is
\begin{equation}
        z_{u,q,N}(t)
        =\zeta_N^{u(t+\eps_N+2q)t},
        \qquad 0\le t<N .                         \label{eq:zc}
\end{equation}
The parameter $q$ contributes only a linear-frequency term in the ordinary Zadoff--Chu variable $t$, but after permutation interleaving it contributes a quadratic term in the external index $k$.  It is therefore retained in the notation until the finite-difference cancellation has been made explicit.
The present work sits at the intersection of CAZAC sequence theory, permutation-polynomial interleaving, and local arithmetic over residue rings.  Its immediate starting point is the QPP-interleaved Zadoff--Chu construction of Berggren and Popovi\'c, who formulated the prime-power inequivalence prediction studied here \cite{BerggrenPopovic2024}; a subsequent analysis by Yuan, Li, and Zeng further investigated the same permutation-interleaved Zadoff--Chu family and established the prime-power sufficiency part of that prediction \cite{YuanLiZeng2026}.  The underlying constant-amplitude zero-autocorrelation sequence model goes back to Chu's polyphase sequences \cite{Chu1972} and to the earlier phase-shift pulse codes of Frank, Zadoff, and Heimiller \cite{FrankZadoff1962}, while Popovi\'c's generalized chirp-like sequences provide a broader CAZAC framework in which Zadoff--Chu sequences appear as a canonical quadratic-phase class \cite{Popovic1992}.  On the interleaving side, Gong's theory of $q$-ary interleaved sequences supplies the general sequence-theoretic viewpoint \cite{Gong1995}, and the use of permutation polynomials over integer rings as structured interleavers was developed in the turbo-code literature by Sun and Takeshita \cite{SunTakeshita2005}.  Takeshita's work on maximum contention-free interleavers \cite{Takeshita2006}, Rivest's characterization of permutation polynomials modulo powers of two \cite{Rivest2001}, the quadratic-inverse analysis of Ryu and Takeshita \cite{RyuTakeshita2006}, and the inverse-degree results of Lahtonen, Ryu, and Suvitie \cite{LahtonenRyuSuvitie2012} together explain why quadratic permutation polynomials over $\Z_N$ are the natural finite-ring objects in the present problem.  The proof techniques used below are also standard in arithmetic finite-algebra contexts: finite-field and residue-ring facts are used in the form recorded by Lidl and Niederreiter \cite{LidlNiederreiter}, elementary valuation and divisibility arguments follow the classical number-theoretic language of Hardy and Wright \cite{HardyWright}, the interpretation of CAZAC phase sums is related to vanishing sums of roots of unity as studied by Lam and Leung \cite{LamLeung2000}, and the broader motivation belongs to the modern theory of low-correlation sequences surveyed by Katz \cite{Katz2018}.
The object studied in \cite{BerggrenPopovic2024} is the interleaved sequence
\begin{equation}
        y_{u,q}^{a,b}(k)
        =z_{u,q,N}(\pi_{a,b}(k)),
        \qquad
        \pi_{a,b}(k)=ak^2+bk\pmod N,               \label{eq:interleaved}
\end{equation}
where $\pi_{a,b}$ is a quadratic permutation polynomial (QPP) over $\Z_N$.  The constant coefficient is normalized to zero, as in the uniqueness computation and equivalence equation in \cite{BerggrenPopovic2024}.  This normalization is also the one recalled in the later work \cite{YuanLiZeng2026}, where the same conjecture is quoted and its sufficiency for prime powers is proved.

The five standard CAZAC-preserving operations considered in \cite{BerggrenPopovic2024} are rotation, translation, decimation, multiplication by a linear-frequency-modulation term, and conjugation.  Thus the relevant equivalence relation is not literal equality of two displayed formulas, but equality after the transformations
\begin{align}
        s(k)&\longmapsto c s(k), & |c|&=1,                                              \label{eq:op-rot}\\
        s(k)&\longmapsto s(k+d), & d&\in\Z_N,                                           \label{eq:op-trans}\\
        s(k)&\longmapsto s(rk), & r&\in\Z_N^\times,                                    \label{eq:op-dec}\\
        s(k)&\longmapsto W_N^{\lambda k}s(k), & \lambda&\in\Z_N,                       \label{eq:op-lfm}\\
        s(k)&\longmapsto \overline{s(k)}.                                               \label{eq:op-conj}
\end{align}
Equivalently, $y_{u,q}^{a,b}$ is Zadoff--Chu-equivalent if, for some
\[
        \sigma\in\{\pm1\},\quad r\in\Z_N^\times,\quad d,\lambda\in\Z_N,
        \quad u'\in\Z_N^\times,
\]
and some unimodular constant $c$, one has
\begin{equation}
        y_{u,q}^{a,b}(k)
        =c\,W_N^{\lambda k}
        \zeta_N^{\sigma u' (rk+d+\eps_N+2q')(rk+d)}             \label{eq:equiv-expanded}
\end{equation}
for all $k\in\Z_N$ and for some $q'\in\Z_N$.

The computational observation in \cite{BerggrenPopovic2024} led to the following prime-power prediction: after excluding the affine case, all QPP-interleaved Zadoff--Chu sequences should be inequivalent to ordinary Zadoff--Chu sequences only when
\begin{equation}
        N=p^n,
        \qquad p>3\text{ prime},
        \qquad n>1.                                  \label{eq:bp-boundary}
\end{equation}
The sufficiency of \eqref{eq:bp-boundary} was later confirmed in \cite{YuanLiZeng2026}.  The present paper proves that \eqref{eq:bp-boundary} is not the correct boundary.  The prime-power assumption is unnecessary; what matters is a local valuation obstruction at a repeated prime $p>3$, together with the absence of the exceptional local collapses at $2$ and $3$.

The main result is the following exact criterion.  For an integer $x$, $\nu_p(x)$ denotes the $p$-adic valuation, with $\nu_p(0)=\infty$, and $p^\alpha\Vert N$ means $p^\alpha\mid N$ but $p^{\alpha+1}\nmid N$.  When $a$ is a residue class modulo $N$ and $0\le t\le\alpha$, the notation $\nu_p(a)\ge t$ means that the image of $a$ in $\Z_{p^\alpha}$ is divisible by $p^t$; in particular, $\nu_p(a)\ge\alpha$ means $a\equiv0\pmod{p^\alpha}$.

\begin{theorem}[local valuation criterion]                                      \label{thm:local}
Let $N\ge3$, let $\pi_{a,b}(k)=ak^2+bk$ be a normalized QPP over $\Z_N$, let $(u,N)=1$, and let $q\in\Z$.  Then the sequence $y_{u,q}^{a,b}$ in \eqref{eq:interleaved} is equivalent to a Zadoff--Chu sequence under \eqref{eq:op-rot}--\eqref{eq:op-conj} if and only if, for each prime power $p^\alpha\Vert N$,
\begin{equation}
        \nu_p(a)\ge \eta(p,\alpha),                  \label{eq:eta-cond}
\end{equation}
where
\begin{equation}
\eta(p,\alpha)=
\begin{cases}
0, & p=2,\ \alpha=1,\\
\alpha-1, & p=2,\ \alpha\ge2,\\
\alpha-1, & p=3,\\
\alpha, & p>3.
\end{cases}                                             \label{eq:eta}
\end{equation}
\end{theorem}

The theorem immediately gives the corrected version of the conjectured boundary for nonzero normalized QPPs.

\begin{corollary}[corrected nonzero-QPP boundary]                         \label{cor:allqpp}
Let
\begin{equation}
        \Q_N=\{\pi_{a,b}(X)=aX^2+bX\in\Z_N[X]:
        \pi_{a,b}\text{ permutes }\Z_N,
        \ a\not\equiv0\pmod N\}.                       \label{eq:QN}
\end{equation}
Then $\Q_N\ne\varnothing$ and every sequence $z_{u,q,N}(\pi(k))$, with $\pi\in\Q_N$, $(u,N)=1$, and $q\in\Z$, is inequivalent to every Zadoff--Chu sequence if and only if
\begin{equation}
        N\text{ is odd},
        \qquad 9\nmid N,
        \qquad \exists p\ge5\text{ prime with }p^2\mid N .             \label{eq:correct-boundary}
\end{equation}
In particular, $N=75=3\cdot5^2$ satisfies \eqref{eq:correct-boundary} but is not a prime power; hence it is a counterexample to the ``only if'' direction of \eqref{eq:bp-boundary}.
\end{corollary}

There is a second natural interpretation in the QPP-interleaver literature: one may discard QPPs whose quadratic term induces an affine permutation.  A standard irreducibility test is
\begin{equation}
        \gcd(N,2a)<N.                                      \label{eq:irr-test}
\end{equation}
Theorem \ref{thm:local} also gives the corresponding exact statement.

\begin{corollary}[irreducible-QPP version]                              \label{cor:irr}
Let
\begin{equation}
        \Qi_N=\{\pi_{a,b}\in\Q_N:\gcd(N,2a)<N\}.             \label{eq:Qirr}
\end{equation}
Then $\Qi_N\ne\varnothing$ and every sequence $z_{u,q,N}(\pi(k))$, with $\pi\in\Qi_N$, is Zadoff--Chu-inequivalent if and only if
\begin{equation}
        9\nmid N
        \quad\text{and}\quad
        \bigl(8\mid N\text{ or }p^2\mid N\text{ for some odd prime }p\bigr).       \label{eq:irr-boundary}
\end{equation}
\end{corollary}

The proof is short at its core, but the details matter because of the modulus $2N$ in \eqref{eq:zc}.  The invariant is the third finite difference of the lifted phase
\begin{equation}
        \Phi_{a,b,q,N}(X)
        =(aX^2+bX+\eps_N+2q)(aX^2+bX).                 \label{eq:Phi}
\end{equation}
Since a quadratic phase has zero third finite difference, the only possible way for a QPP-interleaved Zadoff--Chu phase to become an ordinary Zadoff--Chu phase is for the cubic and quartic parts of \eqref{eq:Phi} to collapse as polynomial functions over the local residue rings.  The collapse occurs exactly at the exceptional valuations in \eqref{eq:eta}; for primes $p>3$ it occurs only when $a$ vanishes modulo $p^\alpha$.

\section{Quadratic Phases and Zadoff--Chu Equivalence}

This section converts the sequence-equivalence problem into a congruence problem for quadratic phase functions.  The conversion is routine, but it is the point at which the factor $2N$ enters.

\begin{lemma}[representative independence]                                  \label{lem:rep}
Let $\eps=\eps_N$.  If $T\equiv T'\pmod N$, then, for every $q\in\Z$,
\begin{equation}
        T(T+\eps+2q)\equiv T'(T'+\eps+2q)\pmod{2N}.          \label{eq:rep-ind}
\end{equation}
Consequently \eqref{eq:zc} is well defined as a function on $\Z_N$.
\end{lemma}

\begin{proof}
Write $T'=T+hN$.  Then
\begin{align}
T'(T'+\eps+2q)-T(T+\eps+2q)
 &=hN(2T+\eps+2q)+h^2N^2 .                         \label{eq:rep-proof-1}
\end{align}
If $N$ is even, then $\eps=0$ and both summands in \eqref{eq:rep-proof-1} are divisible by $2N$.  If $N$ is odd, then $\eps=1$ and
\begin{equation}
T'(T'+1+2q)-T(T+1+2q)
        =hN(2T+1+2q+hN).                                  \label{eq:rep-proof-odd}
\end{equation}
The factor $h(2T+1+2q+hN)$ is even: if $h$ is even this is immediate, while if $h$ is odd then $hN$ is odd and the parenthesis is even.  Hence \eqref{eq:rep-proof-odd} is divisible by $2N$.
\end{proof}

\begin{definition}[admissible quadratic phase]                              \label{def:quad}
A function $F:\Z_N\to\Z_{2N}$ is called an admissible Zadoff--Chu quadratic phase if there are integers $A,B,C$ satisfying
\begin{equation}
        F(k)\equiv Ak^2+Bk+C\pmod{2N}                         \label{eq:quad-phase}
\end{equation}
for all $k\in\Z_N$, with
\begin{equation}
        (A,N)=1,
        \qquad
        B\equiv A\eps_N\pmod 2 .                              \label{eq:quad-parity}
\end{equation}
\end{definition}

The congruence in \eqref{eq:quad-phase} is independent of the integer representative of $k$.  Indeed,
\begin{equation}
A(k+N)^2+B(k+N)+C-(Ak^2+Bk+C)=N(2Ak+AN+B),                 \label{eq:quad-period}
\end{equation}
and the last factor is even when $N$ is even because $B\equiv0\pmod2$, and when $N$ is odd because $B\equiv A\pmod2$.  The parity condition in \eqref{eq:quad-parity} is simply the condition that the exponent in the $\zeta_N$ model has the parity of a Zadoff--Chu phase.  It is automatic after expansion of \eqref{eq:equiv-expanded}, but recording it avoids a common ambiguity between congruences modulo $N$ and congruences modulo $2N$.

\begin{lemma}[quadratic-phase equivalence]                                  \label{lem:quad-equiv}
A unimodular sequence $s(k)=\zeta_N^{F(k)}$ is equivalent to a Zadoff--Chu sequence under \eqref{eq:op-rot}--\eqref{eq:op-conj} if and only if $F$ is an admissible Zadoff--Chu quadratic phase.
\end{lemma}

\begin{proof}
Assume first that $s$ has the form \eqref{eq:equiv-expanded}.  Expanding the exponent gives
\begin{align}
F(k)
 &\equiv C
       +2\lambda k
       +\sigma u'(rk+d+\eps_N+2q')(rk+d)                  \notag\\
 &\equiv C'
       +\bigl(2\lambda+\sigma u'r(2d+\eps_N+2q')\bigr)k
       +\sigma u'r^2 k^2
       \pmod{2N}.                                          \label{eq:quad-from-ops}
\end{align}
Thus $A=\sigma u'r^2$ is a unit modulo $N$, and
\begin{equation}
B-A\eps_N
   \equiv \sigma u'r(2d+\eps_N+2q')-\sigma u'r^2\eps_N
   \pmod 2.                                                \label{eq:parity-check-1}
\end{equation}
Modulo $2$, the right-hand side of \eqref{eq:parity-check-1} is $\sigma u'\eps_N r(1-r)$.  This is zero if $N$ is even, because $\eps_N=0$, and it is also zero if $N$ is odd, because $r(1-r)$ is always even.  Hence \eqref{eq:quad-parity} follows.

Conversely, suppose \eqref{eq:quad-phase}--\eqref{eq:quad-parity} hold.  Take $r=1$, $d=0$, $q'=0$, $\sigma=1$, and take the Zadoff--Chu root $u'$ to be the integer representative $A$.  This is allowed because $(A,N)=1$.  Because $B-A\eps_N$ is even, there exists $\lambda\in\Z_N$ such that
\begin{equation}
        2\lambda \equiv B-A\eps_N\pmod{2N}.                \label{eq:lambda-choice}
\end{equation}
Then
\begin{equation}
        A(k+\eps_N)k+2\lambda k
        \equiv Ak^2+Bk\pmod{2N}.                           \label{eq:reverse-phase-match}
\end{equation}
Absorbing $C$ into the unimodular rotation gives $s(k)$.  Thus the sequence is equivalent to a Zadoff--Chu sequence.
\end{proof}

For the interleaved sequence \eqref{eq:interleaved}, the lifted exponent is
\begin{equation}
        F_{u,a,b,q,N}(k)
        =u\Phi_{a,b,q,N}(k)\pmod{2N},                  \label{eq:F}
\end{equation}
where $\Phi$ is defined in \eqref{eq:Phi}.  Since $u$ is a unit modulo every prime divisor of $N$, the local valuation obstruction is independent of $u$.  The parameter $q$ also disappears from the obstruction because
\begin{equation}
\Phi_{a,b,q,N}(X)
=(aX^2+bX)^2+(\eps_N+2q)(aX^2+bX),                       \label{eq:Phi-split}
\end{equation}
and the second summand is a polynomial of degree at most two in $X$.

\section{Local QPP Conditions}

We shall use only the elementary local form of the standard QPP criterion over integer residue rings.  It follows from the Chinese remainder theorem and the usual prime-power derivative test for permutation polynomials.

\begin{lemma}[local QPP criterion]                                      \label{lem:qpp}
Let $p^\alpha\Vert N$ and let $\pi_{a,b}(X)=aX^2+bX$ be a normalized QPP over $\Z_N$.  Then the following local conditions hold.
\begin{align}
        p\text{ odd}:        &\qquad p\nmid b,
                                      \qquad p\mid a,                         \label{eq:qpp-odd}\\
        p=2,\ \alpha=1:     &\qquad a+b\equiv1\pmod2,                         \label{eq:qpp-two-one}\\
        p=2,\ \alpha\ge2:  &\qquad b\equiv1\pmod2,
                                      \qquad a\equiv0\pmod2.                  \label{eq:qpp-two-high}
\end{align}
Conversely, the congruences \eqref{eq:qpp-odd}--\eqref{eq:qpp-two-high}, imposed for all prime powers $p^\alpha\Vert N$, are sufficient for $\pi_{a,b}$ to permute $\Z_N$.
\end{lemma}

\begin{proof}
By the Chinese remainder theorem, $\pi_{a,b}$ permutes $\Z_N$ if and only if it permutes every $\Z_{p^\alpha}$.  For odd $p$, the reduction modulo $p$ is $aX^2+bX$.  A quadratic polynomial over $\mathbb F_p$ is a permutation only if its quadratic coefficient vanishes and its linear coefficient is nonzero; this gives $p\mid a$ and $p\nmid b$.  Conversely, if $p\mid a$ and $p\nmid b$, then
\begin{equation}
        \pi_{a,b}(x)-\pi_{a,b}(y)
        =(x-y)\bigl(a(x+y)+b\bigr).                       \label{eq:qpp-factor}
\end{equation}
The factor $a(x+y)+b$ is a unit modulo $p^\alpha$, so $\pi_{a,b}(x)\equiv\pi_{a,b}(y)\pmod{p^\alpha}$ implies $x\equiv y\pmod{p^\alpha}$.  Thus the odd-prime conditions are also sufficient.

For $p=2$ and $\alpha=1$, one only needs the map on $\mathbb F_2$ to be nonconstant.  Since $X^2=X$ on $\mathbb F_2$, this is precisely $a+b\equiv1\pmod2$.

For $p=2$ and $\alpha\ge2$, reduction modulo $2$ first gives $a+b\equiv1\pmod2$.  If $a$ were odd, then $b$ would be even and, modulo $4$,
\begin{equation}
        \pi_{a,b}(0)\equiv0,
        \qquad
        \pi_{a,b}(2)=4a+2b\equiv0,                         \label{eq:qpp-two-necessity}
\end{equation}
contradicting injectivity modulo $4$.  Hence $a$ is even and $b$ is odd.  Conversely, if $a$ is even and $b$ is odd, then the factor $a(x+y)+b$ in \eqref{eq:qpp-factor} is odd and therefore a unit modulo $2^\alpha$; hence the same injectivity argument proves that $\pi_{a,b}$ permutes $\Z_{2^\alpha}$.
\end{proof}

The following notation isolates the local threshold in Theorem \ref{thm:local}:
\begin{equation}
\eta(p,\alpha)=
\begin{cases}
0, & p=2,\ \alpha=1,\\
\alpha-1, & p=2,\ \alpha\ge2,\\
\alpha-1, & p=3,\\
\alpha, & p>3.
\end{cases}                                                
\end{equation}
Thus Theorem \ref{thm:local} says that Zadoff--Chu equivalence is completely controlled by the inequalities
\begin{equation}
        \nu_p(a)\ge\eta(p,\alpha_p),
        \qquad p^{\alpha_p}\Vert N.                       \label{eq:eta-all}
\end{equation}

\section{The Finite-Difference Obstruction}

Let $\Delta$ denote the forward difference operator,
\begin{equation}
        (\Delta P)(X)=P(X+1)-P(X),
        \qquad
        \Delta^{j+1}P=\Delta(\Delta^jP).                   \label{eq:Delta}
\end{equation}
For the monomials needed here,
\begin{align}
        \Delta^3 X^2 &=0,                                  \label{eq:d3x2}\\
        \Delta^3 X^3 &=6,                                  \label{eq:d3x3}\\
        \Delta^3 X^4 &=24X+36.                             \label{eq:d3x4}
\end{align}
Combining \eqref{eq:Phi-split} with \eqref{eq:d3x2}--\eqref{eq:d3x4} gives the main identity.

\begin{lemma}[third difference of the interleaved phase]                    \label{lem:d3}
For every $a,b,q,N$,
\begin{equation}
        \Delta^3\Phi_{a,b,q,N}(X)
        =12a(2aX+3a+b).                                    \label{eq:d3phi}
\end{equation}
In particular, the third difference is independent of $q$ and of the parity parameter $\eps_N$.
\end{lemma}

\begin{proof}
From \eqref{eq:Phi-split},
\begin{align}
\Phi_{a,b,q,N}(X)
 &=a^2X^4+2abX^3+b^2X^2+(\eps_N+2q)aX^2+(\eps_N+2q)bX.       \label{eq:Phi-expanded}
\end{align}
The last three terms have degree at most two and hence have zero third difference.  Therefore
\begin{align}
\Delta^3\Phi_{a,b,q,N}(X)
 &=a^2(24X+36)+2ab\cdot6                                      \notag\\
 &=24a^2X+36a^2+12ab                                           \notag\\
 &=12a(2aX+3a+b).                                             \label{eq:d3phi-proof}
\end{align}
\end{proof}

If a phase is equivalent to an ordinary Zadoff--Chu phase, then Lemma \ref{lem:quad-equiv} says it is represented by an admissible quadratic polynomial modulo $2N$.  By \eqref{eq:quad-period}, this representative may be evaluated on arbitrary integer lifts of the residue classes, and its ordinary third forward difference is identically zero modulo $2N$.  Therefore the third difference of the lifted phase must vanish in every local modulus.  For odd primes one may reduce modulo $p^\alpha$, since $2$ is invertible and the extra parity component modulo $2$ contains no $p$-adic information.  For $p=2$, one must work modulo $2^{\alpha+1}$ because the original exponent is defined modulo $2N$.

\begin{proposition}[necessary local inequalities]                          \label{prop:nec}
Let $p^\alpha\Vert N$, let $\pi_{a,b}$ be a normalized QPP over $\Z_N$, and suppose $y_{u,q}^{a,b}$ is Zadoff--Chu-equivalent.  Then
\begin{equation}
        \nu_p(a)\ge\eta(p,\alpha).                         \label{eq:nec-eta}
\end{equation}
\end{proposition}

\begin{proof}
The factor $u$ in \eqref{eq:F} is a unit locally, so it may be ignored for valuation purposes.

First let $p>3$.  By Lemma \ref{lem:qpp}, $p\nmid b$ and $p\mid a$.  If $t=\nu_p(a)<\alpha$, then
\begin{equation}
        3a+b\equiv b\not\equiv0\pmod p,                     \label{eq:pgt3-unit}
\end{equation}
and, since $p\nmid12$,
\begin{equation}
        \nu_p\bigl(\Delta^3\Phi_{a,b,q,N}(0)\bigr)
        =\nu_p(12a(3a+b))
        =t<\alpha.                                          \label{eq:pgt3-nonzero}
\end{equation}
Thus $\Delta^3\Phi_{a,b,q,N}(0)$ is nonzero modulo $p^\alpha$, contradicting the existence of a quadratic representative.  Hence $\nu_p(a)\ge\alpha$, which is \eqref{eq:eta} for $p>3$.

Next let $p=3$.  Again $3\nmid b$ and $3\mid a$.  If $t=\nu_3(a)\le\alpha-2$, then
\begin{equation}
        \nu_3(3a+b)=0,
        \qquad
        \nu_3(12)=1,                                        \label{eq:p3-units}
\end{equation}
and hence
\begin{equation}
        \nu_3\bigl(\Delta^3\Phi_{a,b,q,N}(0)\bigr)
        =1+t<\alpha.                                        \label{eq:p3-nonzero}
\end{equation}
This contradicts local quadraticity modulo $3^\alpha$.  Therefore $\nu_3(a)\ge\alpha-1$.

Now let $p=2$ and $\alpha\ge2$.  By Lemma \ref{lem:qpp}, $b$ is odd and $a$ is even.  If $t=\nu_2(a)\le\alpha-2$, then $3a+b$ is odd and
\begin{equation}
        \nu_2\bigl(\Delta^3\Phi_{a,b,q,N}(0)\bigr)
        =\nu_2(12)+\nu_2(a)
        =2+t<\alpha+1.                                      \label{eq:p2-nonzero}
\end{equation}
Thus the third difference is nonzero modulo $2^{\alpha+1}$, contradicting quadraticity of the lifted phase modulo $2N$.  Hence $\nu_2(a)\ge\alpha-1$.

Finally, if $p=2$ and $\alpha=1$, the threshold is $\eta(2,1)=0$, so there is no condition to prove.
\end{proof}

The proof above explains the exceptional primes.  The factor $12$ in \eqref{eq:d3phi} contributes one power of $3$ and two powers of $2$.  Those powers are exactly what permit a cubic term to masquerade as a linear term modulo $3^\alpha$ or $2^{\alpha+1}$ at the penultimate valuation.

\section{Local Sufficiency}

We now prove that the necessary conditions in Proposition \ref{prop:nec} are sufficient.  The proof is constructive: in every local component the quartic and cubic terms in \eqref{eq:Phi-expanded} reduce to a quadratic function.

\begin{lemma}[the odd-prime reductions]                                   \label{lem:odd-suff}
Let $p^\alpha\Vert N$ be odd, and suppose $\pi_{a,b}$ is a normalized QPP over $\Z_N$.
\begin{enumerate}[label=(\roman*)]
\item If $p>3$ and $\nu_p(a)\ge\alpha$, then
\begin{equation}
        \Phi_{a,b,q,N}(X)
        \equiv b^2X^2+(1+2q)bX
        \pmod{p^\alpha}                                      \label{eq:pgt3-quad}
\end{equation}
when $N$ is odd, and with $1$ replaced by $0$ when $N$ is even.
\item If $p=3$ and $\nu_3(a)\ge\alpha-1$, then $\Phi_{a,b,q,N}(X)$ is congruent modulo $3^\alpha$ to a quadratic polynomial in $X$ whose quadratic coefficient is congruent to $b^2$ modulo $3$.
\end{enumerate}
\end{lemma}

\begin{proof}
The first assertion is immediate from $a\equiv0\pmod{p^\alpha}$ in \eqref{eq:Phi-expanded}.

For $p=3$, if $\alpha=1$, Lemma \ref{lem:qpp} gives $3\mid a$, so $a\equiv0\pmod{3^\alpha}$ and the preceding case applies.  Hence assume $\alpha\ge2$.  The case $a\equiv0\pmod{3^\alpha}$ is again the same.  It remains to consider
\begin{equation}
        a=3^{\alpha-1}c,
        \qquad 3\nmid c.                                    \label{eq:a3-high}
\end{equation}
Then $a^2X^4\equiv0\pmod{3^\alpha}$ because $2\alpha-2\ge\alpha$.  Also
\begin{equation}
        X^3-X=X(X-1)(X+1)\equiv0\pmod3,                     \label{eq:x3x-3}
\end{equation}
so
\begin{equation}
        2abX^3\equiv 2abX\pmod{3^\alpha}.                  \label{eq:cubic-to-linear-3}
\end{equation}
Substituting into \eqref{eq:Phi-expanded} gives
\begin{align}
\Phi_{a,b,q,N}(X)
 &\equiv \bigl(b^2+(\eps_N+2q)a\bigr)X^2
       +\bigl(2ab+(\eps_N+2q)b\bigr)X
       \pmod{3^\alpha}.                                    \label{eq:p3-quad}
\end{align}
The displayed quadratic coefficient is congruent to $b^2\not\equiv0\pmod3$, as required.
\end{proof}

The two-adic component has to be treated modulo $2^{\alpha+1}$, not merely modulo $2^\alpha$.

\begin{lemma}[the two-adic reduction]                                     \label{lem:two-suff}
Let $2^\alpha\Vert N$ with $\alpha\ge2$, and suppose $\pi_{a,b}$ is a normalized QPP over $\Z_N$.  If $\nu_2(a)\ge\alpha-1$, then $\Phi_{a,b,q,N}(X)$ is congruent modulo $2^{\alpha+1}$ to a quadratic polynomial in $X$ whose quadratic coefficient is odd and whose linear coefficient is even.
\end{lemma}

\begin{proof}
If $a\equiv0\pmod{2^\alpha}$, then the claim is immediate because
\begin{equation}
        \Phi_{a,b,q,N}(X)\equiv b^2X^2+2qbX\pmod{2^{\alpha+1}},             \label{eq:a-zero-two}
\end{equation}
and $b$ is odd.  Otherwise write
\begin{equation}
        a=2^{\alpha-1}c,
        \qquad c\text{ odd}.                                \label{eq:a2-high}
\end{equation}
For every integer $X$,
\begin{align}
        X^3-X&=X(X-1)(X+1)\equiv0\pmod2,                   \label{eq:x3x-2}\\
        X^4-X^2&=X^2(X-1)(X+1)\equiv0\pmod4.               \label{eq:x4x2-2}
\end{align}
Therefore
\begin{align}
        2ab(X^3-X)&\equiv0\pmod{2^{\alpha+1}},              \label{eq:cubic-two}\\
        a^2(X^4-X^2)&\equiv0\pmod{2^{\alpha+1}}.            \label{eq:quartic-two}
\end{align}
Indeed, \eqref{eq:cubic-two} has valuation at least $\alpha+1$, and \eqref{eq:quartic-two} has valuation at least $2\alpha\ge\alpha+1$.  Thus
\begin{align}
\Phi_{a,b,q,N}(X)
&=a^2X^4+2abX^3+b^2X^2+2q(aX^2+bX)                         \notag\\
&\equiv (a^2+b^2+2qa)X^2+(2ab+2qb)X
       \pmod{2^{\alpha+1}}.                                \label{eq:p2-quad}
\end{align}
The quadratic coefficient is congruent to $b^2\equiv1\pmod2$.
\end{proof}

For the squarefree two-adic component $2\Vert N$, there is no obstruction.

\begin{lemma}[the squarefree two-adic component]                           \label{lem:two-one}
Assume $2\Vert N$, let $\pi_{a,b}$ be a normalized QPP over $\Z_N$, and let $u$ be odd.  Then $u\Phi_{a,b,q,N}$ modulo $4$ is represented by an admissible quadratic phase with odd quadratic coefficient.
\end{lemma}

\begin{proof}
Put $T(k)=\pi_{a,b}(k)$.  Since $\pi_{a,b}$ permutes $\Z_2$ and $T(0)=0$, one has
\begin{equation}
        T(k)\equiv k\pmod2.                                \label{eq:T-parity-two}
\end{equation}
Because $\eps_N=0$ when $N$ is even,
\begin{equation}
        \Phi_{a,b,q,N}(k)=T(k)\bigl(T(k)+2q\bigr).           \label{eq:Phi-two-one}
\end{equation}
If $k$ is even, then $T(k)$ is even and $T(k)(T(k)+2q)\equiv0\pmod4$.  If $k$ is odd, then $T(k)$ is odd and
\begin{equation}
        T(k)\bigl(T(k)+2q\bigr)
        \equiv 1+2q\pmod4.                                 \label{eq:odd-two-value}
\end{equation}
Thus $u\Phi_{a,b,q,N}(k)\pmod4$ depends only on $k\pmod2$.

Moreover $u\Phi_{a,b,q,N}(0)=0$, and $u\Phi_{a,b,q,N}(1)$ is odd.  Choose an odd $A$ with
\begin{equation}
        A\equiv u\Phi_{a,b,q,N}(1)\pmod2,                   \label{eq:A-mod2}
\end{equation}
and choose an even $B$ so that
\begin{equation}
        A+B\equiv u\Phi_{a,b,q,N}(1)-u\Phi_{a,b,q,N}(0)\pmod4.              \label{eq:B-mod4}
\end{equation}
The quadratic $Ak^2+Bk+u\Phi_{a,b,q,N}(0)$ has odd quadratic coefficient and even linear coefficient.  It also depends only on $k\pmod2$ modulo $4$, because
\begin{equation}
        A(k+2)^2+B(k+2)-Ak^2-Bk
        =4A(k+1)+2B\equiv0\pmod4.                          \label{eq:quad-parity-two}
\end{equation}
It matches $u\Phi_{a,b,q,N}$ at $k=0$ and $k=1$, hence it represents the local phase modulo $4$ for all integers $k$.
\end{proof}

\section{Proof of the Local Criterion}

We now assemble the necessity and sufficiency statements.

\begin{proof}[Proof of Theorem \ref{thm:local}]
Necessity is Proposition \ref{prop:nec}.  For sufficiency, assume \eqref{eq:eta-cond} for every $p^\alpha\Vert N$.

For each odd $p^\alpha\Vert N$, Lemma \ref{lem:odd-suff} gives a local quadratic representative for $u\Phi_{a,b,q,N}$ modulo $p^\alpha$ with unit quadratic coefficient.  For each $2^\alpha\Vert N$ with $\alpha\ge2$, Lemma \ref{lem:two-suff}, after multiplication by the odd integer $u$, gives a quadratic representative modulo $2^{\alpha+1}$ with odd quadratic coefficient and even linear coefficient.  If $2\Vert N$, Lemma \ref{lem:two-one} gives the corresponding representative modulo $4$.

The local representatives have the form
\begin{equation}
        u\Phi_{a,b,q,N}(k)
        \equiv A_p k^2+B_p k+C_p
        \pmod{m_p},                                      \label{eq:local-rep}
\end{equation}
where
\begin{equation}
        m_p=
        \begin{cases}
        p^\alpha, & p\text{ odd and }p^\alpha\Vert N,\\
        2^{\alpha+1}, & p=2\text{ and }2^\alpha\Vert N.
        \end{cases}                                      \label{eq:mp}
\end{equation}
Here $A_p$ is a unit modulo $p$ for every $p\mid N$.  We shall construct integers $A,B,C$ such that
\begin{equation}
        u\Phi_{a,b,q,N}(k)
        \equiv Ak^2+Bk+C\pmod{2N}
        \qquad(k\in\Z_N).                              \label{eq:global-rep}
\end{equation}

If $N$ is even, the moduli $m_p$ in \eqref{eq:mp} are pairwise coprime and their product is $2N$.  The Chinese remainder theorem gives integers $A,B,C$ satisfying
\begin{align}
        A&\equiv A_p\pmod{m_p}, &
        B&\equiv B_p\pmod{m_p}, &
        C&\equiv C_p\pmod{m_p}                         \label{eq:CRT-even}
\end{align}
for every prime $p\mid N$.  Then \eqref{eq:global-rep} holds modulo $2N$.  Moreover $A$ is a unit modulo $N$, and the two-adic representative forces $B\equiv0\pmod2=A\eps_N\pmod2$.

If $N$ is odd, the odd local moduli have product $N$.  First choose $A_0,B_0,C_0$ modulo $N$ by the Chinese remainder theorem from the odd local representatives.  Then choose integers $A,B,C$ satisfying
\begin{align}
        A&\equiv A_0\pmod N,                                             \label{eq:CRT-odd-A}\\
        B&\equiv B_0\pmod N, & B&\equiv A\pmod2,                       \label{eq:CRT-odd-B}\\
        C&\equiv C_0\pmod N, & C&\equiv0\pmod2.                        \label{eq:CRT-odd-C}
\end{align}
This is possible because $(2,N)=1$.  Since $\Phi_{a,b,q,N}(k)=T(k)(T(k)+1+2q)$, with $T(k)=ak^2+bk$, is even-valued when $N$ is odd, $u\Phi_{a,b,q,N}$ and $Ak^2+Bk+C$ agree modulo $2$; by construction they also agree modulo $N$.  Hence \eqref{eq:global-rep} holds modulo $2N$.  Again $(A,N)=1$, and \eqref{eq:CRT-odd-B} gives $B\equiv A\eps_N\pmod2$.

Thus, in both cases, $u\Phi_{a,b,q,N}$ is an admissible Zadoff--Chu quadratic phase.

By Lemma \ref{lem:quad-equiv}, the sequence $\zeta_N^{u\Phi(k)}=y_{u,q}^{a,b}(k)$ is equivalent to a Zadoff--Chu sequence.  This completes the proof.
\end{proof}

The proof also gives an effective reduction algorithm.  For each prime power $p^\alpha\Vert N$, replace the high-valuation terms by
\begin{align}
        a^2X^4&\mapsto0,
        &2abX^3&\mapsto2abX,
        &&(p=3,\ \alpha\ge2,\ \nu_3(a)=\alpha-1),      \label{eq:rule3}\\
        a^2X^4&\mapsto a^2X^2,
        &2abX^3&\mapsto2abX,
        &&(p=2,\ \nu_2(a)=\alpha-1),                 \label{eq:rule2}
\end{align}
and set $a=0$ on every $p>3$ component for which equivalence is possible.  The resulting local quadratics glue by the Chinese remainder theorem.

\section{Corrected Global Boundaries}

We next prove Corollaries \ref{cor:allqpp} and \ref{cor:irr}.  The proofs are included because they expose exactly why the prime-power condition is too restrictive.

\begin{proof}[Proof of Corollary \ref{cor:allqpp}]
Assume first that \eqref{eq:correct-boundary} holds.  Let $\pi_{a,b}\in\Q_N$.  Since $N$ is odd, Lemma \ref{lem:qpp} gives
\begin{equation}
        p\mid a\quad\text{for every }p\mid N.              \label{eq:rad-divides-a}
\end{equation}
If $p\Vert N$ is squarefree, then \eqref{eq:rad-divides-a} already gives $\nu_p(a)\ge1=\alpha$.  Thus squarefree prime factors never create a nonzero quadratic component.  Because $a\not\equiv0\pmod N$, there must be a repeated prime power $\ell^\beta\Vert N$ for which
\begin{equation}
        \nu_\ell(a)<\beta .                                \label{eq:nonzero-repeated}
\end{equation}
The hypotheses $N$ odd and $9\nmid N$ exclude $\ell=2$ and $\ell=3$ with $\beta\ge2$.  Hence $\ell>3$.  For this prime, \eqref{eq:eta} requires $\nu_\ell(a)\ge\beta$ for equivalence, contradicting \eqref{eq:nonzero-repeated}.  By Theorem \ref{thm:local}, $y_{u,q}^{a,b}$ is inequivalent to every Zadoff--Chu sequence.

Conversely, suppose \eqref{eq:correct-boundary} fails.  If $N$ is even, put
\begin{equation}
        a=N/2.                                             \label{eq:even-witness-a}
\end{equation}
Choose $b$ by the Chinese remainder theorem so that
\begin{align}
        b&\equiv0\pmod2, &&\text{if }2\Vert N,             \label{eq:even-b1}\\
        b&\equiv1\pmod{2^\alpha}, &&\text{if }2^\alpha\Vert N,\ \alpha\ge2, \label{eq:even-b2}\\
        b&\equiv1\pmod{p^\alpha}, &&\text{if }p^\alpha\Vert N,\ p\text{ odd}. \label{eq:even-b3}
\end{align}
Then \eqref{eq:qpp-odd}--\eqref{eq:qpp-two-high} show that $\pi_{a,b}$ is a normalized QPP, and $a\not\equiv0\pmod N$.  Moreover
\begin{equation}
        \nu_2(a)=
        \begin{cases}
        0, &2\Vert N,\\
        \alpha-1, &2^\alpha\Vert N,\ \alpha\ge2,
        \end{cases}                                      \label{eq:even-witness-val}
\end{equation}
and $a\equiv0$ on every odd prime-power component.  Thus \eqref{eq:eta-cond} holds everywhere, and Theorem \ref{thm:local} says the corresponding QPP-interleaved sequence is Zadoff--Chu-equivalent.  Hence the universal inequivalence property fails.

If $N$ is odd and $9\mid N$, put
\begin{equation}
        a=N/3,
        \qquad b=1.                                      \label{eq:three-witness}
\end{equation}
Then $\pi_{a,1}$ is a normalized QPP, $a\not\equiv0\pmod N$, and
\begin{equation}
        \nu_3(a)=\alpha_3-1,
        \qquad
        \nu_p(a)\ge\alpha_p\quad(p\ne3).                 \label{eq:three-witness-val}
\end{equation}
Again \eqref{eq:eta-cond} holds everywhere, so Theorem \ref{thm:local} gives an equivalent interleaved sequence.

It remains to consider odd $N$ with $9\nmid N$ and with no prime $p\ge5$ satisfying $p^2\mid N$.  Then every odd prime divisor of $N$ is squarefree, except possibly $3$ to the first power.  Lemma \ref{lem:qpp} forces $p\mid a$ for every prime $p\mid N$, hence $N\mid a$.  Therefore $\Q_N=\varnothing$, and the non-vacuous property in Corollary \ref{cor:allqpp} is false by definition.  This completes the proof.
\end{proof}

\begin{proof}[Proof of Corollary \ref{cor:irr}]
First, \eqref{eq:irr-test} can hold for some normalized QPP if and only if
\begin{equation}
        8\mid N\quad\text{or}\quad p^2\mid N\text{ for some odd prime }p.     \label{eq:irr-nonempty}
\end{equation}
Indeed, Lemma \ref{lem:qpp} forces $a\equiv0\pmod p$ on every odd local component.  Such a component can contribute to \eqref{eq:irr-test} exactly when its exponent is at least two, since then one may have $\nu_p(a)=1<\alpha$.  On the two-adic side, if $2\Vert N$ then $2a\equiv0\pmod2$, and if $4\Vert N$ then Lemma \ref{lem:qpp} forces $a$ even, so $2a\equiv0\pmod4$; neither case can make $\gcd(N,2a)<N$.  If $2^\alpha\Vert N$ with $\alpha\ge3$, one may take $\nu_2(a)=1$, giving $\nu_2(2a)=2<\alpha$.  Combining these local choices with the required QPP divisibility conditions at all other primes, and then choosing $b$ by Lemma \ref{lem:qpp} and the Chinese remainder theorem, gives the converse existence statement.

If $9\mid N$, the witness $a=N/3$, $b=1$ from \eqref{eq:three-witness} satisfies
\begin{equation}
        \gcd(N,2a)=N/3<N,                                 \label{eq:three-irr}
\end{equation}
and is Zadoff--Chu-equivalent by Theorem \ref{thm:local}.  Hence the universal irreducible-QPP inequivalence property fails.

Assume conversely that $9\nmid N$ and \eqref{eq:irr-nonempty} holds.  Let $\pi_{a,b}\in\Qi_N$.  Since $\gcd(N,2a)<N$, there is a prime power $p^\alpha\Vert N$ such that
\begin{equation}
        \nu_p(2a)<\alpha .                                \label{eq:irr-local}
\end{equation}
If $p=2$, then \eqref{eq:irr-local} gives $\nu_2(a)\le\alpha-2$, which violates \eqref{eq:eta-cond}.  If $p$ is odd, then \eqref{eq:irr-local} gives $\nu_p(a)<\alpha$.  Since $9\nmid N$, the case $p=3$ cannot have $\alpha\ge2$; for a squarefree $3$-component no irreducibility can arise because Lemma \ref{lem:qpp} forces $3\mid a$.  Thus either $p=2$ with the two-adic violation above, or $p>3$ with $\nu_p(a)<\alpha$, also a violation of \eqref{eq:eta-cond}.  Theorem \ref{thm:local} now implies that every irreducible QPP-interleaved sequence is Zadoff--Chu-inequivalent.
\end{proof}

\section{Explicit Counterexamples and Arithmetic Audit}

The smallest non-prime-power length satisfying the corrected nonzero-QPP boundary \eqref{eq:correct-boundary} is
\begin{equation}
        N=75=3\cdot5^2.                                  \label{eq:N75}
\end{equation}
For this length, the local QPP criterion gives
\begin{equation}
        a\in\{15,30,45,60\}\pmod{75},
        \qquad b\in\Z_{75}^\times.                       \label{eq:N75-qpp}
\end{equation}
For every such $a$,
\begin{equation}
        \nu_5(a)=1<2,
        \qquad 5\nmid b,
        \qquad 5\nmid 3a+b,                              \label{eq:N75-obstruction-1}
\end{equation}
and therefore
\begin{equation}
        \nu_5\bigl(\Delta^3\Phi_{a,b,q,75}(0)\bigr)
        =\nu_5\bigl(12a(3a+b)\bigr)=1<2.                 \label{eq:N75-obstruction-2}
\end{equation}
Thus no such phase can be quadratic modulo $25$, and no corresponding interleaved sequence can be equivalent to a Zadoff--Chu sequence.  Since
\begin{equation}
        |\Z_{75}^\times|=75\left(1-\frac13\right)\left(1-\frac15\right)=40, \label{eq:phi75}
\end{equation}
there are
\begin{equation}
        4\cdot40=160                                      \label{eq:N75-count}
\end{equation}
nonzero normalized QPP coefficient pairs, all Zadoff--Chu-inequivalent.

This is not an isolated accident.  The corrected condition \eqref{eq:correct-boundary} contains every length of the form
\begin{equation}
        N=s\prod_{i=1}^{r}p_i^{\alpha_i},                 \label{eq:family-general}
\end{equation}
where
\begin{align}
        &s\text{ is odd and squarefree, and }(s,p_1\cdots p_r)=1,          \label{eq:family-s}\\
        &3^2\nmid s\prod_{i=1}^{r}p_i^{\alpha_i},          \label{eq:family-no9}\\
        &p_1,\ldots,p_r\text{ are distinct primes with }p_i\ge5,
        \qquad \alpha_i\ge2
        \qquad(1\le i\le r),                              \label{eq:family-good}
\end{align}
and $r\ge1$.  Prime powers $p^\alpha$ with $p\ge5$ are only the special case $s=1$ and $r=1$.

The following table records several small lengths.  The counts are not used in the proof; they are included to make clear where the corrected criterion differs from the prime-power prediction.

\begin{table}[!t]
\centering
\caption{Small normalized nonzero-QPP lengths.  ``All inequivalent'' is the conclusion of Corollary \ref{cor:allqpp}.}
\label{tab:audit}
\begin{tabular}{c c c c}
\toprule
$N$ & factorization & $|\Q_N|$ & conclusion \\
\midrule
$25$  & $5^2$          & $80$   & all inequivalent \\
$49$  & $7^2$          & $252$  & all inequivalent \\
$75$  & $3\cdot5^2$   & $160$  & all inequivalent; not a prime power \\
$121$ & $11^2$         & $1100$ & all inequivalent \\
$125$ & $5^3$          & $2400$ & all inequivalent \\
$147$ & $3\cdot7^2$   & $504$  & all inequivalent; not a prime power \\
$175$ & $5^2\cdot7$   & $480$  & all inequivalent; not a prime power \\
\midrule
$45$  & $3^2\cdot5$   & $48$   & not all; witness $a=15,b=1$ \\
$50$  & $2\cdot5^2$   & $180$  & not all; witness $a=25$ and $b\equiv0\pmod2$ \\
$98$  & $2\cdot7^2$   & $546$  & not all; witness $a=49$ and $b\equiv0\pmod2$ \\
\bottomrule
\end{tabular}
\end{table}

For example, at $N=45$ the bad local component is $3^2$.  With
\begin{equation}
        a=15=N/3,
        \qquad b=1,                                      \label{eq:N45-witness}
\end{equation}
Theorem \ref{thm:local} predicts equivalence.  Locally modulo $9$,
\begin{align}
        (15X^2+X)^2+(1+2q)(15X^2+X)
        &\equiv X^2+30X^3+(1+2q)(6X^2+X)\pmod9              \notag\\
        &\equiv X^2+3X+(1+2q)(6X^2+X)\pmod9,              \label{eq:N45-collapse}
\end{align}
where $30X^3\equiv3X$ follows from $X^3\equiv X\pmod3$.  The quartic term vanishes modulo $9$ because $15^2\equiv0\pmod9$.  Thus the phase is locally quadratic, and the $5$-component is affine because $15\equiv0\pmod5$.

At $N=50$, the bad component is not a repeated odd prime but the squarefree two-adic component.  Taking $a=25$ and choosing $b$ with
\begin{equation}
        b\equiv0\pmod2,
        \qquad b\equiv1\pmod{25},                         \label{eq:N50-b}
\end{equation}
produces a QPP.  The $25$-component is affine because $a\equiv0\pmod{25}$, while the component modulo $2$ cannot obstruct quadraticity.  This explains why the corrected nonzero-QPP boundary requires $N$ to be odd, not merely to have a repeated prime $p\ge5$.

\section{Root Index, Shift Parameter, and Inverse QPPs}

The local criterion is insensitive to the Zadoff--Chu root index $u$.  Indeed, every local obstruction in Proposition \ref{prop:nec} is an assertion that
\begin{equation}
        \Delta^3\Phi_{a,b,q,N}(0)\not\equiv0\pmod{p^\alpha}
        \quad(p\text{ odd})                               \label{eq:root-insens-odd}
\end{equation}
or
\begin{equation}
        \Delta^3\Phi_{a,b,q,N}(0)\not\equiv0\pmod{2^{\alpha+1}}
        \quad(p=2).                                       \label{eq:root-insens-two}
\end{equation}
Multiplication by $u\in\Z_N^\times$ preserves each nonvanishing statement.  Therefore the criterion holds simultaneously for all roots.

The parameter $q$ is similarly harmless, but for a different reason.  It changes the lifted phase by
\begin{equation}
        2q(ak^2+bk),                                      \label{eq:q-change}
\end{equation}
which is already a quadratic correction at the level of finite differences.  More explicitly,
\begin{equation}
        \Delta^3\bigl(2q(ak^2+bk)\bigr)=0.                \label{eq:q-d3-zero}
\end{equation}
Thus the same theorem applies to every cyclic shift parameter $q$ in the Zadoff--Chu family.

One may also consider interleaving by a permutation whose inverse is a QPP.  Let $\rho$ be any permutation of $\Z_N$ such that
\begin{equation}
        \rho^{-1}(X)=aX^2+bX\pmod N.                      \label{eq:inverse-qpp}
\end{equation}
The CAZAC property of $z_{u,q,N}(\rho(k))$ may follow from the QPP structure of $\rho^{-1}$, as in \cite{BerggrenPopovic2024}.  However, Zadoff--Chu equivalence of $z_{u,q,N}(\rho(k))$ is not governed directly by the coefficient $a$ in \eqref{eq:inverse-qpp}; it is governed by the lifted phase of $\rho(k)$ itself.  If $\rho$ has degree larger than two as a polynomial representative, then the finite-difference invariant must be applied to that representative, not to its inverse.  This is why Theorem \ref{thm:local} is stated for QPP interleavers and not for inverse-QPP interleavers.  The method, however, extends: if a polynomial interleaver $P(k)$ is used, then Zadoff--Chu equivalence forces
\begin{equation}
        \Delta^3\bigl((P(k)+\eps_N+2q)P(k)\bigr)\equiv0     \label{eq:general-method}
\end{equation}
in every local phase modulus.  If $P$ has degree $d\le1$, this third difference is identically zero.  If $d\ge2$, its formal degree over the integers is at most $2d-3$, and it may drop further after reduction as a polynomial function modulo prime powers.  Thus local valuation collapses become a finite arithmetic problem in polynomial functions over prime-power residue rings.

\section{Why the Prime-Power Pattern Appeared}

The data in \cite{BerggrenPopovic2024} correctly identified the prime powers
\begin{equation}
        25,
        \quad 49,
        \quad 121,
        \quad 125,
        \quad 169,
        \quad 343                                      \label{eq:bp-prime-powers}
\end{equation}
as lengths for which all tested normalized QPP interleavings were inequivalent to ordinary Zadoff--Chu sequences.  These lengths all satisfy \eqref{eq:correct-boundary}.  The missed phenomenon is that squarefree odd factors other than an additional power of $3$ do not alter the obstruction.  Algebraically, this is the statement
\begin{equation}
        p\Vert N,
        \quad p\text{ odd}
        \quad\Longrightarrow\quad
        p\mid a\quad\Longrightarrow\quad a\equiv0\pmod p,  \label{eq:squarefree-noeffect}
\end{equation}
so a squarefree odd component contributes only an affine local phase.  The obstruction must come from a repeated component, and if one repeated component $p^\alpha$ with $p>3$ remains nonzero, the third difference detects it through
\begin{equation}
        \Delta^3\Phi(0)=12a(3a+b),                         \label{eq:detect-repeat}
\end{equation}
with
\begin{equation}
        \nu_p(12)=0,
        \qquad
        \nu_p(3a+b)=0.                                    \label{eq:detect-repeat-val}
\end{equation}
This is independent of whether the rest of $N$ is $1$, $3$, $7$, $11$, or any odd squarefree factor prime to $p$.

The exceptional repeated prime $3$ is different because
\begin{equation}
        \nu_3(12)=1.                                      \label{eq:nu3-12}
\end{equation}
At the penultimate valuation $\nu_3(a)=\alpha-1$, the third difference vanishes modulo $3^\alpha$, and the identity
\begin{equation}
        X^3\equiv X\pmod3                                  \label{eq:fermat3}
\end{equation}
turns the cubic term into a linear term.  Similarly, the exceptional two-adic behavior is caused by
\begin{equation}
        \nu_2(12)=2,
        \qquad
        X^3\equiv X\pmod2,
        \qquad
        X^4\equiv X^2\pmod4.                              \label{eq:two-identities}
\end{equation}
These congruences are precisely the exceptional cases encoded in \eqref{eq:eta}.

\section{Divisor Calculus and a Non-Enumerative Test}

The local criterion can be rewritten as a single divisibility condition on the quadratic coefficient.  This form is useful because it separates the existence of nonzero QPPs from their possible collapse to ordinary Zadoff--Chu sequences.
For $p^\alpha\Vert N$, define the QPP divisibility exponent
\begin{equation}
        \mu(p,\alpha)=
        \begin{cases}
        0, & p=2,\ \alpha=1,\\
        1, & p=2,\ \alpha\ge2,\\
        1, & p\text{ odd},
        \end{cases}                                      \label{eq:mu}
\end{equation}
and let $\eta(p,\alpha)$ be the exponent in \eqref{eq:eta}.  Put
\begin{align}
        M_N&=\prod_{p^\alpha\Vert N}p^{\mu(p,\alpha)},     \label{eq:MN}\\
        E_N&=\prod_{p^\alpha\Vert N}p^{\eta(p,\alpha)}.    \label{eq:EN}
\end{align}
Then Lemma \ref{lem:qpp} says that a normalized QPP coefficient $a$ must satisfy
\begin{equation}
        M_N\mid a,                                      \label{eq:MN-divides-a}
\end{equation}
with the additional local unit or parity condition on $b$, while Theorem \ref{thm:local} says that Zadoff--Chu equivalence is exactly
\begin{equation}
        E_N\mid a.                                      \label{eq:EN-divides-a}
\end{equation}
The divisor $E_N$ has the particularly simple closed form
\begin{equation}
        E_N=\frac{N}{2^{\mathbf 1_{2\mid N}}3^{\mathbf 1_{3\mid N}}},        \label{eq:EN-closed}
\end{equation}
where a missing prime contributes exponent $0$.  Indeed, the full exponent $\alpha$ is retained for every $p>3$, while one exponent is removed from a present $2$- or $3$-part.

Let
\begin{equation}
        L_N=\operatorname{lcm}(M_N,E_N).                 \label{eq:LN}
\end{equation}
There exists a nonzero normalized QPP coefficient $a$ whose interleaving is Zadoff--Chu-equivalent if and only if
\begin{equation}
        L_N<N.                                           \label{eq:lcm-test}
\end{equation}
The implication is immediate: $M_N\mid a$ and $E_N\mid a$ if and only if $L_N\mid a$, and a nonzero residue class modulo $N$ divisible by $L_N$ exists precisely when $L_N<N$.  Conversely, if $L_N<N$, take $a=L_N$ and choose $b$ locally by
\begin{equation}
        b\equiv
        \begin{cases}
        0\pmod2, & 2\Vert N\text{ and }a\equiv1\pmod2,\\
        1\pmod2, & 2\Vert N\text{ and }a\equiv0\pmod2,\\
        1\pmod{2^\alpha}, & 2^\alpha\Vert N,\ \alpha\ge2,\\
        1\pmod{p^\alpha}, & p^\alpha\Vert N,\ p\text{ odd},
        \end{cases}                                      \label{eq:lcm-b-choice}
\end{equation}
then combine these congruences by the Chinese remainder theorem.  The pair $(a,b)$ is a normalized QPP and satisfies \eqref{eq:EN-divides-a}.

The exponent of $p$ in $L_N$ is
\begin{equation}
        \lambda(p,\alpha)=\max\{\mu(p,\alpha),\eta(p,\alpha)\}.              \label{eq:lambda}
\end{equation}
Explicitly,
\begin{equation}
\lambda(p,\alpha)=
\begin{cases}
0,        & p=2,\ \alpha=1,\\
\alpha-1, & p=2,\ \alpha\ge2,\\
1,        & p=3,\ \alpha=1,\\
\alpha-1, & p=3,\ \alpha\ge2,\\
\alpha,   & p>3.
\end{cases}                                               \label{eq:lambda-table}
\end{equation}
Therefore
\begin{equation}
        L_N=N                                             \label{eq:LN-equals-N}
\end{equation}
if and only if $N$ is odd and $9\nmid N$.  Under this condition, the nonzero normalized QPP coefficient set is nonempty if and only if $M_N<N$, namely if and only if at least one prime power $p^\alpha\Vert N$ with $p\ge5$ has $\alpha\ge2$.  Thus the lcm test alone recovers the corrected non-vacuous boundary
\begin{equation}
        N\text{ odd},\qquad 9\nmid N,\qquad \exists p\ge5:\ p^2\mid N.        \label{eq:lcm-correct-boundary}
\end{equation}
No search over roots, shifts, LFM parameters, or decimations is needed.

For odd $N$ the coefficient count also has a closed form.  Since
\begin{equation}
        M_N=\prod_{p\mid N}p=\operatorname{rad}(N),        \label{eq:MN-rad-odd}
\end{equation}
one has
\begin{equation}
        a=\operatorname{rad}(N)c\pmod N,
        \qquad 0\le c<\frac{N}{\operatorname{rad}(N)},     \label{eq:a-rad-c}
\end{equation}
and the nonzero condition removes only $c=0$.  Since $b\in\Z_N^\times$, the number of normalized nonzero QPP coefficient pairs is
\begin{equation}
        |\Q_N|=\varphi(N)\left(\frac{N}{\operatorname{rad}(N)}-1\right)
        \qquad(N\text{ odd}).                             \label{eq:odd-count}
\end{equation}
For instance,
\begin{align}
        |\Q_{75}|&=\varphi(75)(75/15-1)=40\cdot4=160,      \label{eq:count-75}\\
        |\Q_{147}|&=\varphi(147)(147/21-1)=84\cdot6=504,   \label{eq:count-147}\\
        |\Q_{175}|&=\varphi(175)(175/35-1)=120\cdot4=480.  \label{eq:count-175}
\end{align}
These are all non-prime-power lengths covered by \eqref{eq:lcm-correct-boundary}.

\section{Sharpness of the Local Thresholds}

Within the QPP class, the thresholds in \eqref{eq:eta} are sharp in the following local sense.  Write $t=\nu_p(a)$ and use the QPP condition that the appropriate linear coefficient is a local unit.
For $p>3$, Lemma \ref{lem:qpp} gives $p\nmid b$.  If $t<\alpha$, then
\begin{equation}
        3a+b\equiv b\not\equiv0\pmod p,                  \label{eq:sharp-p-unit}
\end{equation}
and hence
\begin{equation}
        \nu_p\bigl(\Delta^3\Phi_{a,b,q,N}(0)\bigr)
        =\nu_p(12)+\nu_p(a)+\nu_p(3a+b)
        =t<\alpha.                                        \label{eq:sharp-p-val}
\end{equation}
Thus any weakening of the local vanishing condition $a\equiv0\pmod{p^\alpha}$ is visible already at $k=0$.  Conversely, if $\nu_p(a)\ge\alpha$, then $ak^2+bk\equiv bk\pmod{p^\alpha}$, and the local phase is quadratic.  Since $a\equiv0\pmod{p^\alpha}$ is already the zero condition in $\Z_{p^\alpha}$, there is no stronger residue-class condition to impose at this local component.

For $p=3$ and $\alpha\ge2$, the same computation gives
\begin{equation}
        \nu_3\bigl(\Delta^3\Phi_{a,b,q,N}(0)\bigr)=t+1     \label{eq:sharp-3-val}
\end{equation}
whenever $t<\alpha$ and $3\nmid b$.  Therefore $t\le\alpha-2$ is impossible for a quadratic phase.  At the boundary $t=\alpha-1$, the obstruction disappears, and the phase collapses because
\begin{align}
        a^2X^4&\equiv0\pmod{3^\alpha},                    \label{eq:sharp-3-quartic}\\
        2abX^3&\equiv2abX\pmod{3^\alpha}.                 \label{eq:sharp-3-cubic}
\end{align}
The second congruence is exactly $X^3-X\equiv0\pmod3$ multiplied by a coefficient of valuation $\alpha-1$.  Hence lowering the threshold admits non-quadratic phases, while raising it to $\alpha$ would exclude genuine equivalent boundary cases.  For $p=3$ and $\alpha=1$, Lemma \ref{lem:qpp} already forces $a\equiv0\pmod3$, so the theorem imposes no additional local obstruction.

For $p=2$ and $\alpha\ge2$, QPPs have $b$ odd and $a$ even.  If $t\le\alpha-2$, then $3a+b$ is odd and
\begin{equation}
        \nu_2\bigl(\Delta^3\Phi_{a,b,q,N}(0)\bigr)=t+2<\alpha+1,              \label{eq:sharp-2-val}
\end{equation}
so the phase is not quadratic modulo $2^{\alpha+1}$.  At $t=\alpha-1$ the identities
\begin{align}
        2ab(X^3-X)&\equiv0\pmod{2^{\alpha+1}},             \label{eq:sharp-2-cubic}\\
        a^2(X^4-X^2)&\equiv0\pmod{2^{\alpha+1}}             \label{eq:sharp-2-quartic}
\end{align}
show that the cubic and quartic parts reduce to linear and quadratic terms.  Thus the two-adic threshold is exactly $\alpha-1$.  When $2\Vert N$, Lemma \ref{lem:two-one} shows that the two-adic component has no obstruction; imposing a positive two-adic valuation would be too strong.

These computations also explain why a finite search up to a modest bound can be misleading.  A repeated $p>3$ component is detected at $k=0$ by
\begin{equation}
        \Delta^3\Phi(0)=a\cdot 12(3a+b),                  \label{eq:detected-factor}
\end{equation}
where $12(3a+b)$ has $p$-adic valuation zero.  Repeated $2$- and $3$-components lose one local exponent because $12$ is divisible by $2^2\cdot3$ and because $X^3-X$ has universal divisibility by $2$ and $3$.  The global classification is therefore controlled by the simple arithmetic deletion
\begin{equation}
        N\longmapsto E_N=\frac{N}{2^{\mathbf 1_{2\mid N}}3^{\mathbf 1_{3\mid N}}}.          \label{eq:deletion-map}
\end{equation}

\section{Consequences for Sequence Design}

The corrected boundary changes the available design space.  The prime-power condition \eqref{eq:bp-boundary} gives lengths
\begin{equation}
        N=p^\alpha,
        \qquad p\ge5,
        \qquad \alpha\ge2.                               \label{eq:design-primepower}
\end{equation}
Corollary \ref{cor:allqpp} permits the larger class
\begin{equation}
        N=m p^\alpha,
        \qquad p\ge5,
        \qquad \alpha\ge2,
        \qquad m\text{ odd squarefree},
        \qquad (m,p)=1,
        \qquad 9\nmid m.                                  \label{eq:design-larger}
\end{equation}
More generally, it permits any odd $N$ with no repeated factor $3^2$ and with at least one repeated prime $p\ge5$.

For such lengths, every nonzero normalized QPP coefficient $a$ has a nonvanishing component at some repeated $p>3$.  Consequently, every sequence in the QPP-interleaved family considered in \cite{BerggrenPopovic2024} that is generated by a nonzero normalized QPP lies outside the ordinary Zadoff--Chu equivalence class.  When the external CAZAC result for that construction is invoked, this gives QPP-interleaved CAZAC sequences outside the ordinary Zadoff--Chu class.  This statement does not assert pairwise inequivalence among different QPP interleavers; it asserts inequivalence to all ordinary Zadoff--Chu sequences under \eqref{eq:op-rot}--\eqref{eq:op-conj}.  Pairwise equivalence among different interleavers is a separate problem, because two distinct quadratic phases may have the same square after reduction modulo $2N$ or may be related by central symmetries of the base Zadoff--Chu sequence.

The finite-difference method gives a quick screening rule.  For a candidate QPP $\pi_{a,b}$, compute
\begin{equation}
        D_{a,b}=12a(3a+b).                                \label{eq:screen}
\end{equation}
If there exists $p^\alpha\Vert N$ with
\begin{equation}
        \begin{cases}
        p>3,       & \nu_p(D_{a,b})<\alpha,\\
        p=3,       & \nu_3(D_{a,b})<\alpha,\\
        p=2,       & \nu_2(D_{a,b})<\alpha+1,
        \end{cases}                                      \label{eq:screen-local}
\end{equation}
then the interleaved phase cannot be Zadoff--Chu-equivalent.  For QPPs, Lemma \ref{lem:qpp} simplifies \eqref{eq:screen-local} to the threshold \eqref{eq:eta}; for higher-degree interleavers, \eqref{eq:general-method} provides the analogous screening invariant.

\section{Conclusion}

The equivalence problem for normalized QPP-interleaved Zadoff--Chu sequences is local.  The exact local invariant is the third finite difference of the lifted phase, and it yields the valuation criterion \eqref{eq:eta}.  The prime-power lengths $p^n$ with $p>3$ and $n>1$ are only a subfamily of the true nonzero-QPP inequivalence lengths.  The corrected condition is
\[
        N\text{ odd},
        \qquad 9\nmid N,
        \qquad p^2\mid N\text{ for some }p\ge5.
\]
Thus $N=75$ is already a complete counterexample to the conjectured prime-power-only boundary, and Theorem \ref{thm:local} gives the full replacement criterion.
\section*{Declaration of Generative AI and AI-Assisted Technologies in the Writing Process}
During the preparation of this work, the authors used DeepSeek to build a specialized agent for solving mathematical problems, which was employed to generate an initial proof of the main theorem. After using this tool, the authors reviewed and edited the content as needed and take full responsibility for the content of the published article.

\end{document}